\pdfoutput=1
\documentclass[11pt, a4paper]{amsart}
\usepackage{fixltx2e}
\usepackage{amsfonts}
\usepackage{amssymb}
\usepackage{amsmath}
\usepackage{mathtools}
\usepackage[T1]{fontenc}
\usepackage[bookmarks=true,pdfdisplaydoctitle=true]{hyperref}
\hypersetup{
pdfpagemode=UseNone,
pdfstartview=FitH,
pdftitle={Sparse domination on non-homogeneous spaces with an application to Ap weights},
pdfauthor={Alexander Volberg and Pavel Zorin-Kranich},
pdflang=en-US
}

\usepackage[safeinputenc=true,style=alphabetic,firstinits,maxnames=5,doi=false,eprint=true,isbn=false,url=false,backend=bibtex]{biblatex}

\newbibmacro{string+doiurl}[1]{%
  \iffieldundef{doi}{%
    \iffieldundef{url}{%
         #1%
      }{%
      \href{\thefield{url}}{#1}%
    }%
  }{%
    \href{http://dx.doi.org/\thefield{doi}}{#1}%
  }%
}
\DeclareFieldFormat[article,misc]{title}{\usebibmacro{string+doiurl}{\mkbibquote{#1}}}

\newcommand{\bR}{\mathbb{R}}
\newcommand{\bC}{\mathbb{C}}

\newcommand{\cS}{\mathcal{S}}
\newcommand{\cF}{\mathcal{F}}
\newcommand{\cC}{\mathcal{C}}
\newcommand{\cD}{\mathcal{D}}

\newcommand{\1}{\mathbf{1}}
\DeclarePairedDelimiter\abs{\lvert}{\rvert}
\DeclarePairedDelimiter\norm{\lVert}{\rVert}
\DeclarePairedDelimiter\Set\{\}

\newcommand\dist{\operatorname{dist}}
\newcommand\supp{\operatorname{supp}}
\newcommand{\cz}{Calder\'{o}n--Zygmund }
\newcommand{\dif}{\mathrm{d}}
\newcommand{\Dini}{\mathrm{Dini}}

\numberwithin{equation}{section}
\newtheorem{theorem}[equation]{Theorem}
\newtheorem{cor}[equation]{Corollary}
\newtheorem{lemma}[equation]{Lemma}
\theoremstyle{definition}

\newtheorem{defin}[equation]{Definition}

\begin{document}

\title[Domination, non-homogeneous]{Sparse domination on non-homogeneous spaces with an application to $A_p$ weights}
\author[A. Volberg]{Alexander Volberg}
\thanks{AV is partially supported by the NSF grant DMS-1600065 and by the Hausdorff Institute for Mathematics, Bonn, Germany}
\author[P. Zorin-Kranich]{Pavel Zorin-Kranich}


\begin{abstract}
We extend Lerner's recent approach to sparse domination of \cz operators to upper doubling (but not necessarily doubling), geometrically doubling metric measure spaces.
Our domination theorem is different from the one obtained recently by Conde-Alonso and Parcet and yields a weighted estimate with the sharp power $\max(1,1/(p-1))$ of the $A_{p}$ characteristic of the weight.
\end{abstract}

\maketitle

\section{Introduction}
\label{two}

We are interested in weighted estimates for an operator $T$ with a kernel of \cz type (see below) acting on functions on a metric space $X$ with a non-doubling measure $\mu$.
We always assume that the operator $T$ is bounded on $L^2(X, \mu)$; such operators are called (non-homogeneous) \cz operators.

In the case $X=\bR^{d}$ with the Lebesgue measure, the linear in $A_2$ characteristic estimate for \cz operators, formerly called the ``$A_2$ conjecture'', took some efforts of a large group of mathematicians to settle.
For the Ahlfors--Beurling transform and the Hilbert transform this has been done in \cite{MR1894362} and \cite{MR2354322}, respectively, using a Bellman function approach.
This linear $A_{2}$ estimate for the Ahlfors--Beurling transform had an important application to the theory of quasiregular maps \cite{MR1815249}.
After a number of intermediate results, of which we would like to mention the beautiful papers by Cruz-Uribe, Martell, and P\'erez \cite{MR2628851} and by Lacey, Petermichl, and Reguera \cite{MR2657437}, where the $A_{2}$ conjecture has been proved for dyadic singular operators, the $A_2$ conjecture has been finally proved in full generality by Hyt\"onen \cite{MR2912709}.
Shortly thereafter, a proof based on the methods of \cite{MR2407233} has been obtained in \cite{MR3176607}, and an extension to doubling measure spaces has been obtained in \cite{MR3188553}.
This required the construction of ``random dyadic lattices'' of Christ cubes on doubling metric measure spaces.

The above mentioned proofs are based on decompositions of \cz operators into dyadic singular operators (martingale shifts).
Martingale shifts with respect to doubling measures are ``good'' in the sense that their weighted norms grow linearly in the $A_2$ characteristic of the weight.
In the non-homogeneous situation we hit a very serious difficulty on this path, described in the articles by L\'opez-S\'anchez, Martell, and Parcet \cite{MR3269175} and Thiele, Treil, and Volberg \cite{MR3406523}: for non-doubling measures $\mu$ there is a huge class of martingale shifts that are not good.
So, if one wants to proceed by this path, such ``dangerous'' martingale shifts should be completely avoided in the decomposition of the \cz operator.

A class of good martingale shifts, called ``$L^1(\mu)$-normalized shifts'', has been identified in \cite{MR3539383}.
In \cite{MR3406523} it has been shown that any martingale transform is a good shift; a short proof of this result has been found by Lacey \cite{arXiv:1501.05818}.
In \cite[Theorem 2.11]{MR3269175} an interesting characterization of weak type $(1,1)$ for martingale shifts has been given.
Shifts that have weak type $(1,1)$ are good, and the $L^1(\mu)$-normalized martingale shifts from \cite{MR3539383} form a subclass of the good shifts found in \cite{MR3269175}.

A more recent approach to sharp weighted estimates for \cz operators consists in estimating (rather than representing) them by \emph{sparse operators}
\begin{equation}
\label{dom0}
A_{\cS}f = \sum_{Q\in \cS} \langle \abs{f}\rangle_Q \1_Q,
\qquad
\langle f\rangle_Q := \frac{1}{\mu(Q)} \int_I f \dif\mu,
\end{equation}
where $\cS$ is a \emph{sparse} family of cubes (that is, for every $Q\in\cS$ there exists a subset $E(Q)$ such that $\mu(E(Q)) \geq \frac12 \mu(Q)$ and the sets $E(Q)$ are pairwise disjoint).
The pervasive importance of sparse operators has been realized by Lerner, who proved in \cite{MR3127380} that sparse operators control \cz operators ``on average'' in the sense that $\norm{Tf}_{X} \lesssim_{T} \sup_{\cS} \norm{A_{\cS}f}_{X}$ holds for every Banach function space $X$.
Pointwise estimates for $\abs{Tf}$ by convex combinations of sparse operators have been later obtained in the works of Lerner and Nazarov \cite{arXiv:1508.05639}, Conde-Alonso and Rey \cite{MR3521084}, Lacey \cite{arXiv:1501.05818}, and Lerner \cite{MR3484688}.
The ``sparse operator approach'' has also been successfully applied to other classes of operators, not necessarily given in integral form, see \cite{MR3531367}.
These results, combined with the short proof of weighted estimates for sparse operators by Moen \cite{MR3000426}, provide the most concise proofs of the $A_{2}$ conjecture.

Lerner's local mean oscillation approach \cite{MR2721744} to quantitative weighted estimates has been recently extended to non-doubling measures by Conde-Alonso and Parcet \cite{arXiv:1604.03711}.
Their result is that $Tf$ is dominated by the composition of a certain sparse operator with a certain maximal operator.
This gives a weighted estimate for the operator norm of $T$, but for doubling measures $\mu$ it grows as $[w]_{A_2}^{2}$ on $L^{2}(w \dif\mu)$, and hence does not recover the sharp weighted bound in this classical setting.

In this article we extend Lerner's version \cite{MR3484688} of Lacey's sparse domination algorithm \cite{arXiv:1501.05818} to non-doubling measures.
The stopping time construction in Lerner's article works well as long as the starting cube is doubling in the sense of \eqref{eqdob22}.
However, the stopping cubes need not be doubling.
We have tried to deal with this difficulty using Tolsa's \cz decomposition with respect to a non-homogeneous measure (that has been found for the purpose of proving that a non-homogeneous \cz operator $T$ is weakly bounded on $L^1(\mu)$, a result previously proved without \cz decomposition in \cite{MR1626935}).
But Tolsa's \cz decomposition uses cubes with bounded overlap, unlike the classical one, which uses disjoint cubes.
This seems to lead to an uncontrollable growth of overlap when the decomposition is recursively iterated.

We avoid these problems by working with David--Mattila cells that substitute the dyadic grid and have convenient properties (stated in Lemma~\ref{DM}).

\begin{theorem}
\label{Lnh}
Let $(X,d,\mu)$ be an upper doubling, geometrically doubling metric measure space and $\alpha\geq 200$.
Then for every $L^{2}$ bounded \cz operator $T$ on $X$, every bounded set $X'\subset X$, and every integrable function $f$ supported on $X'$ we can find sparse families $\cF_{n}$, $n=0,1,\dots$, of David--Mattila cells such that the estimate
\[
T^{\sharp}f
\lesssim_{T,\alpha}
\sum_{n=0}^{\infty} 100^{-n} \sum_{Q\in \cF_{n}} \frac{\int_{30B(Q)} \abs{f} \dif\mu}{\mu(\alpha B(Q))} \cdot \1_Q
\]
holds pointwise $\mu$-almost everywhere on $X'$.
\end{theorem}

From this result one can easily deduce the following bound for the maximally truncated operator $T^{\sharp}$ on the weighted space $L^p(w \dif\mu)$, see Section~\ref{cons}.
\begin{cor}
\label{A2}
Let $(X,d,\mu)$ be an upper doubling, geometrically doubling metric measure space and let  $T$ be a \cz operator on $X$.
Then for every $1<p<\infty$ we have
\begin{equation}
\label{eq:dep-on-weight}
\norm{T^{\sharp}}_{L^{p}(w)\to L^{p}(w)}
\lesssim_{\alpha}
\sup_{Q\in\cD} \frac{\sigma(200 B(Q)) w(Q) \sigma(Q)^{(p-2)_{+}} w(Q)^{(p'-2)_{+}}}{\mu(\alpha B(Q)) \mu(Q)^{p^{*}-1}},
\end{equation}
where $\sigma=w^{-1/(p-1)}$ is the dual weight, $p^{*}=\max(p,p')$, and the supremum is taken over David--Mattila cells.
\end{cor}
For doubling measures $\mu$ the right-hand side of \eqref{eq:dep-on-weight} is comparable to the usual sharp power $[w]_{A_{p}}^{\max(1,1/(p-1))}$ of the $A_{p}$ characteristic of the weight.
On the other hand, it is not clear whether for general measures our estimate is stronger than the one in \cite{arXiv:1604.03711}.
Also, it is up for debate what the most appropriate definition of the $A_{p}$ constants for weights with respect to non-doubling measures should be.
Although David--Mattila cells seem to have the same geometric structure as Christ's cubes in spaces of homogeneous type, which have been characterized in \cite{MR3113086}, a definition in terms of this rather large class of sets does not seem completely satisfactory.

\section{Notation and preliminaries}
\label{nota}

\subsection{Upper doubling measures and \cz operators}
\begin{defin}[{\cite{MR2675934}}]
A metric measure space $(X,d,\mu)$ is called \emph{upper doubling} if there exists a \emph{dominating function} $\lambda:X\times (0,\infty) \to (0,\infty)$ and a constant $C_{\lambda}>0$ such that for every $x\in X$ the function $r\mapsto \lambda(x,r)$ is nondecreasing and
\[
\mu(B(x,r)) \leq \lambda(x,r) \leq C_{\lambda} \lambda(x,r/2)
\]
holds for all $x\in X$ and $r\in (0,\infty)$.
\end{defin}
If $\lambda$ is a dominating function, then by \cite[Proposition 1.3]{MR2943664}
\[
\tilde\lambda(x,r) := \inf_{z\in X} \lambda(z,r+d(x,z))
\]
is also a dominating function (with the same constant $C_{\lambda}$) that is not larger than the original dominating function $\lambda$ and has the additional property that
\begin{equation}
\label{eq:df-loc}
\tilde\lambda(x,r) \leq C_{\lambda}\tilde\lambda(y,r)
\text{ for all }
x,y\in X
\text{ with }
d(x,y) \leq r.
\end{equation}
We will assume from now on that $\lambda$ satisfies \eqref{eq:df-loc}.

\begin{defin}
\label{ud}
A metric space $(X,d)$ is called \emph{geometrically doubling} (with doubling dimension $n$) if for every $R\geq r>0$ and every ball $B$ of radius $R$ the cardinality of an $r$-separated subset of $B$ can be at most $C(R/r)^{n}$.
\end{defin}

\begin{defin}
\label{czONud}
A \emph{\cz kernel} on a geometrically doubling, upper doubling metric measure space $(X,d,m)$ is a map $K:X\times X\setminus \Delta \to \bC$ such that
\begin{equation}
\label{eq:czk-size}
\abs{K(x,y)} \leq \frac{C_{K}}{\lambda(x,d(x,y))}
\end{equation}
for some $C_{K}\geq 0$ and all $x,y\in X$, $x\neq y$, and
\begin{equation}
\label{eq:czk-smoothness}
\abs{K(x,y)-K(x',y)}+\abs{K(y,x)-K(y,x')} \leq \omega\Big( \frac{d(x,x')}{d(x,y)} \Big) \frac{1}{\lambda(x,d(x,y))}
\end{equation}
for all $x,x',y\in X$ with $d(x,x')<\frac12 d(x,y)$, where $\omega : [0,\infty)\to [0,\infty)$ is a Dini modulus of continuity, that is, a monotonically increasing subadditive function with $\omega(x)=0 \iff x=0$ and $\norm{\omega}_{\Dini} := \sum_{j\geq 0} \omega(2^{-j}) < \infty$.

A \emph{\cz operator} with kernel $K$ is a linear operator $T$ such that for all bounded functions $f$ with bounded support the restriction of $Tf$ to the complement of the support of $f$ is given by
\[
Tf(x) = \int K(x,y) f(y) \dif\mu(y),
\quad
x \not\in \supp f.
\]
The $\epsilon$-truncation of $T$ is defined by
\[
T_{\epsilon}f(x) = \int_{d(x,y)>\epsilon} K(x,y) f(y) \dif\mu(y),
\]
note that this integral converges absolutely for every $f\in L^{p}(X,\mu)$, $1\leq p<\infty$, and every $x\in X$.
The maximally truncated operator $T^{\sharp}$ is defined by
\[
T^{\sharp}f(x) := \sup_{\epsilon>0} \abs{T_{\epsilon}f(x)}.
\]
\end{defin}
It has been proved in \cite{MR2957235} that $L^{2}$ boundedness of a \cz operator $T$ implies that both $T$ and the maximally truncated operator $T^{\sharp}$ have weak type $(1,1)$ (the results in that article are stated for power moduli of continuity $\omega(t)=ct^{\tau}$, $0<\tau\leq 1$, but the proofs only use the Dini condition).
An alternative proof appears in \cite{MR3023861}.
The former proof extends the proof for power bounded measures in \cite{MR1626935} and the latter the proof in \cite{MR1812821}.

\subsection{David--Mattila cells}
\label{DM}
Now we will consider the dyadic lattice of ``cubes'' with small boundaries of David--Mattila associated with $\mu$.
This lattice has been constructed in \cite[Theorem 3.2]{MR1768535} for non-doubling measures on $\bR^{n}$, and the proof works without alterations for general geometrically doubling metric spaces.
Its properties are summarized in the next lemma.

\begin{lemma}[David, Mattila]
\label{lemcubs}
Let $(X,d)$ be a geometrically doubling metric space with doubling dimension $n$ and let $\mu$ be a locally finite Borel measure on $X$.
Consider two constants $C_0>1$ and $A_0>5000\,C_0$ and denote $W=\supp\mu$.
Then there exists a sequence of partitions of $W$ into Borel subsets $Q$, $Q\in \cD_k$, with the following properties:
\begin{itemize}
\item For each integer $k\geq0$, $W$ is the disjoint union of the ``cubes'' $Q$, $Q\in\cD_k$, and
if $k<l$, $Q\in\cD_l$, and $R\in\cD_k$, then either $Q\cap R=\varnothing$ or else $Q\subset R$.

\item The general position of the cubes $Q$ can be described as follows.
For each $k\geq0$ and each cube $Q\in\cD_k$, there is a ball $B(Q)=B(z_Q,r(Q))$ such that
\[z_Q\in W, \qquad A_0^{-k}\leq r(Q)\leq C_0\,A_0^{-k},\]
\[W\cap B(Q)\subset Q\subset W\cap 28\,B(Q)=W \cap B(z_Q,28r(Q)),\]
and
\[\mbox{the balls\, $5B(Q)$, $Q\in\cD_k$, are disjoint.}\]

\item The cubes $Q\in\cD_k$ have small boundaries.
That is, for each $Q\in\cD_k$ and each
integer $l\geq0$, set
\[N_l^{ext}(Q)= \Set{x\in W\setminus Q:\,\dist(x,Q)< A_0^{-k-l}},\]
\[N_l^{int}(Q)= \Set{x\in Q:\,\dist(x,W\setminus Q)< A_0^{-k-l}},\]
and
\[N_l(Q)= N_l^{ext}(Q) \cup N_l^{int}(Q).\]
Then
\begin{equation}\label{eqsmb2}
\mu(N_l(Q))\leq (C^{-1}C_0^{-3n-1}A_0)^{-l}\,\mu(90B(Q)).
\end{equation}

\item Denote by $\cD_k^{db}$ the family of cubes $Q\in\cD_k$ for which
\begin{equation}\label{eqdob22}
\mu(100B(Q))\leq C_0\,\mu(B(Q)).
\end{equation}
For the cubes $Q\in\cD_k\setminus \cD_k^{db}$ we have that $r(Q)=A_0^{-k}$ and
\begin{equation}\label{eqdob23}
\mu(cB(Q))\leq C_0^{-1}\mu(100c B(Q))\quad
\end{equation}
for all $1\leq c\leq C_{0}$.
\end{itemize}
\end{lemma}

We use the notation $\cD=\bigcup_{k\geq0}\cD_k$.
Observe that the families $\cD_k$ are only defined for $k\geq0$.
So the diameters of the cubes from $\cD$ are uniformly bounded from above.
For $Q\in\cD_{k}$ we call the cube $\hat Q\in \cD_{k-1}$ such that $\hat Q\supset Q$ the parent of $Q$.
We denote
$\cD^{db}=\bigcup_{k\geq0}\cD_k^{db}$.

\section{Sparse domination}
\subsection{Grand maximal truncation}
We put for a cell $Q\in \cD$ and $x\in Q$
\[
F(x, Q) := \int_{X\setminus 30B(Q)} K(x, y)f(y) \dif\mu(y)\,.
\]
We consider the (localized) grand maximal truncated operator
\[
N_{Q_0}f(x) := \1_{Q_{0}}(x) \sup_{x\in P,\,P \in \cD(Q_0)} \sup_{y\in P} \abs{F(y, P)},
\quad
Q_0\in \cD
\]
that has been introduced in \cite{MR3484688}.
We claim that the operator $N_{Q_{0}}$ has weak type $(1,1)$.
Indeed, let $x,x'\in Q\in \cD_{k}$.
Then
\begin{align*}
\abs{F(x,Q) - F(x',Q)}
&\leq
\int_{X\setminus 30 B(Q)} \abs{K(x,y)-K(x',y)} \abs{f(y)} \dif\mu(y)\\
&\lesssim
\sum_{j\geq 0} \int_{\dist(y,Q) \sim 2^{j} r(Q)} \frac{\omega(56r(Q)/(2^{j}r(Q)))}{\lambda(x,2^{j} r(Q))} \abs{f(y)} \dif\mu(y)\\
&\lesssim
\norm{\omega}_{\Dini} M_{\lambda} f(x),
\end{align*}
where
\[
M_{\lambda} f(x) := \sup_{R>0} \frac{1}{\lambda(x,R)} \int_{B(x,R)} \abs{f} \dif\mu.
\]

Moreover, for any $r\sim r(Q)$ we have
\begin{equation}
\label{eq:compare-truncations}
\begin{aligned}
\abs{T_{r}f(x) - F(x, Q)}
&\leq
\int_{30B(Q) \Delta B(x,r)} \abs{K(x,y)} \abs{f(y)} \dif\mu(y)\\
&\lesssim
C_{K} M_{\lambda} f(x).
\end{aligned}
\end{equation}
Therefore we have the pointwise inequality
\begin{equation}
\label{eq:N-tsharp}
\abs{N_{Q_{0}}f - T^{\sharp}f} \leq C(\norm{\omega}_{\Dini} + C_{K}) M_{\lambda} f
\quad\text{on}\quad Q_{0},
\end{equation}
valid for functions supported on $30B(Q_{0})$,
and this implies that $N_{Q_{0}}$ has weak type $(1,1)$ with a constant independent of $Q_{0}$.

\subsection{Consecutive scales}
For a David--Mattila cell $Q$ and a large number $\alpha\geq 200$, we denote 
\[
A(f, Q) :=\frac1{\mu(\alpha B(Q))} \int_{30 B(Q)}\abs{f} \dif\mu
\]
and
\[
\Theta(Q) := \frac{\mu(\alpha B(Q))}{\lambda(z_{Q},\alpha r(Q))}.
\]
Notice that if $A_{0}$ is chosen large enough and $Q \subset \hat Q$ are two nested  cells, then
\begin{equation}
\label{in}
30 B(Q)\subset 30 B(\hat Q),
\end{equation}
even though the centers of these two balls are different.

For every cell $Q\in\cD$ and every $x\in Q$ we have
\begin{equation}
\label{eq:consecutive-scales-radius-av}
\begin{aligned}
\int_{30 B(\hat{Q}) \setminus 30 B(Q)} \abs{K(x,y)} \abs{f(y)} \dif\mu(y)
&\leq
\frac{C_{K}}{\lambda(x,r(Q))} \int_{30 B(\hat{Q})} \abs{f} \dif\mu\\
&\lesssim
\frac{1}{\lambda(x,\alpha r(\hat Q))} \int_{30 B(\hat{Q})} \abs{f} \dif\mu\\
&\lesssim
\Theta(\hat Q) A(f,\hat Q),
\end{aligned}
\end{equation}
where $\hat Q\in\cD$ denotes the parent of $Q$, and in particular
\begin{equation}
\label{eq:2scales}
N_{\hat Q}(f\1_{30B(\hat Q)}) (x)
\leq
C \Theta(\hat Q) A(f,\hat Q) + N_{Q}(f\1_{30 B(Q)})(x).
\end{equation}

This is useful because the numbers $\Theta(Q)$ are bounded by $1$ and decay exponentially fast along nested sequences of non-doubling cubes.
\begin{lemma}[{cf.~\cite[Lemma 5.31]{MR1768535}}]\label{lemcad22}
Let $l_{0}$ be the maximal number with $100^{l_{0}}\leq C_{0}/\alpha$ and suppose that $C_{0}^{l_{0}/2} > C_{\lambda}^{\lceil \log_{2}A_{0} \rceil}$, where $C_{\lambda}$ is the doubling constant of the dominating function.
Let
\[
Q_{0}=\hat Q_{1} \supset Q_{1} = \hat Q_{2} \supset \dots
\]
be a nested family of cubes such that $Q_{1},Q_{2},\dots$ are non-doubling.
Then
\begin{equation}\label{eqdk88}
\Theta(Q_{k})\lesssim C_0^{-k l_{0}/2}\Theta(Q_{0}).
\end{equation}
\end{lemma}
\begin{proof}
This follows from \eqref{eqdob23}.
\end{proof}

\subsection{The cube selection procedure}
The main part of the proof of Theorem \ref{Lnh} is a recursive cube selection procedure.
Since we are using non-sharp truncations (i.e.\, the grand maximal function associated to a ball is applied to the restriction of the function $f$ to a larger ball), Lacey's stopping time argument \cite{arXiv:1501.05818} only works well for doubling cubes, but yields stopping cubes that are in general non-doubling.
When we arrive at a non-doubling stopping cube, we simply keep subdividing it into smaller cubes until we hit a doubling cube.
The contributions of the intermediate scales turn out to shrink exponentially, and we obtain the following result.

\begin{lemma}
\label{lem:recursion}
Let $Q_{0} \in \cD^{db}$ be a doubling cube and $f$ be an integrable function supported on $30B(Q_{0})$.
Then there exists a subset $\Omega \subset Q_{0}$, collections of pairwise disjoint cubes $\cC_{n}(Q_{0})\subset\cD$, $n=1,\dots$, contained in $\Omega$, and a collection of pairwise disjoint doubling cubes $\cF(Q_{0})\subset\cD^{db}$ contained in $\Omega$ with the following properties:
\begin{enumerate}
\item\label{lem:recursion:minor} $\mu(\Omega) \leq \frac12 \mu(Q_{0})$,
\item For every $P\in \cF$ and $Q\in \cC_{n}$ we have either $P\subset Q$ or $P\cap Q=\emptyset$,
\item Almost everywhere we have the estimate
\begin{multline}
\label{eq:recursion}
N_{Q_{0}} (\1_{30B(Q_{0})}f) \1_{Q_{0}}
\leq
\sum_{P\in\cF(Q_{0})} N_{P} (\1_{30B(P)}f) \1_{P}\\
+
C A(f,Q_{0}) \1_{Q_{0}}
+ C \sum_{n=1}^{\infty} 100^{-n} \sum_{Q\in\cC_{n}(Q_{0})} A(f,Q) \1_{Q}.
\end{multline}
\end{enumerate}
\end{lemma}

\begin{proof}
Note that the maximal function
\begin{equation}
\label{eq:Mmu}
M_{\mu}f(x) := \sup_{x\in Q} A(f,Q)
\end{equation}
has weak type $(1,1)$ by the Vitali covering lemma.
Let $K>0$ be so large that the weak type $(1,1)$ inequalities for $N_{Q_{0}}$ and $M_{\mu}$ will imply that the bad set
\[
\Omega:=\Set{x\in Q_0:  N_{Q_0}f(x) > K A(f, Q_0)}\cup \Set{x\in Q_0:  M_\mu f(x) > K A(f, Q_0)}
\]
has measure bounded by $\frac12 \mu(Q_0)$ (this is the only step in which we use the doubling property of $Q_{0}$).
The set $\Omega$ is the disjoint union of the maximal cells contained in it.
Let us call this family of cells $\cC_{0}(Q_0)$.
They are not necessarily doubling.

By definition of $\Omega$ we have
\begin{equation}
\label{vne}
x\in Q_0 \setminus \cup_{Q\in \cC_{0}(Q_0)} Q
\implies
N_{Q_0}f(x) \le K A(f, Q_0)\,.
\end{equation}

Now we want the estimate of $N_{Q_0}f(x)$ for $x$ in each $Q\in \cC_{0}(Q_0)$.
By maximality of $Q$ we know that
\[
\sup_{x\in Q} \abs{F(x, \tilde Q)}
\leq
\sup_{y\in \tilde Q} \abs{F(y, \tilde Q)}
\leq
K A(f, Q_0)
\]
for every $\tilde Q$ with $Q \subsetneq \tilde Q \subseteq Q_{0}$.
Hence
\[
N_{Q_{0}}f(x) \leq K A(f, Q_0) + N_{\hat Q}(f1_{30 B(\hat Q)})(x),
\]
where $\hat Q$ is the parent of $Q$ in $\cD$.
Applying \eqref{eq:2scales} on the right-hand side, using maximality of $Q$ to estimate $A(f, \hat Q) \leq K A(f, Q_0)$, and using the fact that $\Theta(\hat Q)\leq 1$ we obtain
\[
N_{Q_{0}}f(x) \leq C A(f, Q_0) + N_{Q}(f1_{30 B(Q)})(x)
\]
with some larger value of $C$.

The families $\cC_{n}$ are now constructed inductively as follows.
Put all doubling cubes in $\cC_{0}$ into $\cF$ and let $\cC_{1}$ consist of the remaining non-doubling cubes.
Suppose that $\cC_{n}$, $n\geq 1$, has already been constructed.
Then we put every $Q\in\cD$ such that $\hat Q\in\cC_{n}$ into $\cF$ if it is doubling and into $\cC_{n+1}$ if it is non-doubling.

In view of \cite[Lemma 5.28]{MR1768535} the chain
\[
Q_{1} \supset Q_{2} \supset \dotsb \supset Q_{N} \ni x,
\quad
Q_{n} \in \cC_{n},
\]
terminates after finitely many term for almost every $x\in\Omega$.
If $C_{0}$ is sufficiently large, then Lemma~\ref{lemcad22} yields $\Theta(Q_{n}) \lesssim 100^{-n}$, and we obtain the claim~\eqref{eq:recursion} summing the estimate \eqref{eq:2scales} over the cubes $Q_{1},\dotsc,Q_{N}$.
\end{proof}

\subsection{Proof of Theorem~\ref{Lnh}}
We begin by finding a ball $B\subset X$ that contains $X'$ and satisfies the doubling condition $\mu(100 B) \leq C_{0} \mu(B)$.
Rescaling the metric by a constant we may assume that this ball has radius $C_{0}$.
The construction in \cite[Theorem 3.2]{MR1768535} now yields a system of David--Mattila cells such that $X'$ is contained in some cell $Q_{0}$.

Recursive application of Lemma~\ref{lem:recursion} now yields an estimate of the required form for $N_{Q_{0}}f$ as follows.
We initialize the collection of doubling cubes $\cF_{0}^{0}:=\Set{Q_{0}}$.
Given a collection of doubling cubes $\cF_{0}^{k}$, an application of Lemma~\ref{lem:recursion} to each cube $P\in\cF_{0}^{k}$ yields collections of non-doubling cubes $\cC_n^{P}, n=1,\dots$ and a collection of doubling cubes $\cF^{P}$.
We define
\[
\cF_n^{k+1} := \cup_{P \in \cF_{0}^{k}} \cC_n^{P},\ n=1,\dots,
\text{ and }
\cF_{0}^{k+1} := \cup_{P \in \cF^{k}} \cF^{P}.
\]
It follows from part \eqref{lem:recursion:minor} of Lemma~\ref{lem:recursion} that for every $k\geq 0$, every $n\geq 0$, and every $R\in \cF_n^k$ we have
\[
\sum_{Q \in \cF_n^{k+1} : Q\subset R} \mu(Q) \le \frac12 \mu(R).
\]
We denote $\cF_n:=\cup_k \cF_n^k$, and these are precisely the sparse families we need in Theorem~\ref{Lnh}.

In view of \eqref{eq:N-tsharp} it remains to estimate the maximal function $M_{\lambda}$ by a sparse operator.
To this end note that every ball $B(x,R)$ is contained in $30 B(Q)$ for some cell $Q$ with $r(Q)\lesssim_{A_{0}} R$.
It follows that
\[
M_{\lambda}f(x)
\lesssim
\sup_{x\in Q\in\cD} \lambda(x,r(Q))^{-1} \int_{30 B(Q)} \abs{f} \dif\mu
\lesssim
\sup_{x\in Q\in\cD} \Theta(Q) A(f,Q).
\]
Lemma~\ref{lem:recursion} continues to hold with $N$ replaced by the localized maximal operator $\tilde N_{Q_{0}}f(x) := \sup_{x\in Q\in\cD(Q_{0})} \Theta(Q) A(f,Q)$ with identical proof.
This provides the required sparse domination for $M_{\lambda}$.

\section{Consequences for weighted estimates}
\label{cons}

In view of Theorem~\ref{Lnh}, in order to prove Corollary~\ref{A2} it suffices to obtain the corresponding estimate for a sparse operator
\[
Tf = \sum_{Q\in\cS} A(f,Q) \1_{Q},
\]
where $\cS$ is a sparse collection of David--Mattila cells.
We repeat the proofs in \cite{MR3000426} and \cite{MR3484688}.
By duality the norm of $T$ as an operator on $L^{p}(w)$ is equal to the best constant in the inequality
\[
\abs{\int T(f\sigma) gw} \leq K \norm{f}_{L^{p}(\sigma\dif\mu)} \norm{g}_{L^{p'}(w\dif\mu)}.
\]
Let $E(Q)\subset Q\in\cS$ be disjoint subsets such that $\mu(E(Q))>\frac12 \mu(Q)$.
Then
\begin{align*}
\abs[\big]{\int T(f\sigma) gw \dif\mu}
&\leq
\sum_{Q} A(f\sigma,Q) \int_{Q} \abs{gw} \dif\mu\\
&=
\sum_{Q} \frac{\sigma(200 B(Q)) w(Q)}{\mu(\alpha B(Q))} \Big( \frac{1}{\sigma(200 B(Q))} \int_{30 B(Q)} \abs{f\sigma} \dif\mu \Big) \Big( \frac{1}{w(Q)} \int_{Q} \abs{gw} \dif\mu \Big) \\
&\lesssim
  \Big( \sup_{Q} \frac{\sigma(200 B(Q)) w(Q)}{\mu(\alpha B(Q)) \sigma(E(Q))^{1/p} w(E(Q))^{1/p'}} \Big)\\
&\cdot  \Big( \sum_{Q} \Big( \frac{1}{\sigma(200 B(Q))} \int_{30 B(Q)} \abs{f\sigma} \dif\mu \Big)^{p} \sigma(E(Q)) \Big)^{1/p}\\
&\cdot  \Big( \sum_{Q} \Big( \frac{1}{w(Q)} \int_{Q} \abs{gw} \dif\mu \Big)^{p'} w(E(Q))\Big)^{1/p'}.
\end{align*}
The last two terms are estimated by the $M_{\sigma \dif\mu}$ maximal function of $f$ (defined in \eqref{eq:Mmu}) and the martingale maximal function $M_{w \dif\mu}^{\cD}$ with measure $w\dif\mu$ of $g$, respectively.

In order to estimate the supremum over $Q$ in the first term by the right-hand side of \eqref{eq:dep-on-weight} note that $w^{1/p'}\sigma^{1/p}\equiv \1$, so that $\mu(Q) \lesssim \mu(E(Q)) \leq w(E(Q))^{1/p}\sigma(E(Q))^{1/p'}$ by H\"older's inequality.
Taking this inequality to power $p^{*}-1$ and using it in the denominator we obtain the claim.

\section{The  $A_2$ conjecture for arbitrary  non-homogeneous \cz operators in dimension $1$}
\label{dim1}

It is nice to notice that the $A_2$ characteristic in our main result (Corollary~\ref{A2}, where the supremum is taken over David--Mattila cells) becomes almost the usual $A_2$ characteristic if we consider \cz operators with respect to an arbitrary measure in $\bR^1$.

To notice that let us observe first of all that we have never used the ``small boundary property'' of David--Mattila cells from Section \ref{DM} in the proof.
We will now indicate how the construction in \cite[Section 3]{MR1768535} can be modified in the case of a non-atomic measure $\mu$ with compact support on the real line in such a way that
\begin{enumerate}
\item the resulting cells will be intervals and
\item the cells will satisfy all conditions of Section \ref{DM} except may be the small boundary requirement.
\end{enumerate}
The restriction to non-atomic measures does not lose generality in application to $A_2$ questions.

We consider the balls (in our case intervals, as we are on $\bR^1$) $B(x)$ built on pages 145--146 of \cite[Theorem 3.2]{MR1768535}.
Then we find the discrete subset $I^0$ of points $x$ such that $5B(x)$ are disjoint and $25B(x)$ cover the support $E=\supp\mu$.
Now we swerve a little bit from the path of \cite[Theorem 3.2]{MR1768535}, and we construct balls (intervals) $B_4^{0}$ as follows.
We allow each $5B(x)$, $x\in I^0$, to extend beyond its end-points to the left and to the right with the speed proportional to the size of $B(x)$.
We stop the extension when the earliest of the following happen: 1) the extension reaches an end-point of $25B(x)$, 2) it meets another extension.
Notice that the extensions beyond the left and the right end-point of a given~$5B(x)$ can stop for different reasons.

Notice that we did not use the notions of $B_1(x), B_2(x), B_3(x)$ of \cite[Theorem 3.2]{MR1768535}, but rather immediately built disjoint $B_4(x)= B_4^{0}(x), x\in I=I^{0}$.
As in Lemma 3.33 of \cite[Theorem 3.2]{MR1768535}, we can claim that $B_4(x), x\in I$ are disjoint and cover $E$.

Then we apply the preceding construction to each scale $A_0^{-k}$.
We get $B_4^k(x), x\in I_k$, which is again a disjoint covering of $E$.
Next we wish to replace $B_4^k$'s by a finer version, by taking unions of $B_4^m(y)$, $m>k$.

We need the supervising relation called $h$ in  \cite[Theorem 3.2]{MR1768535}.
For each $k\ge 1$ the point $y\in I^k$ will be supervised by $x\in I^{k-1}$ if and only if $y\in B_4^{k-1}(x)$.

Notice that supervising relationship is {\it monotone}, meaning that if $x_1<x_2, x_i\in I^{k-1}, i=1,2$, and $y_i\in I^k$ is a supervisee of $x_i$, $i=1,2$ correspondingly, then $y_1<y_2$.

Introduce (as in the paper of David--Mattila) for $x\in I^k$ the set
\[
D^k_\ell := \cup_{y\in I^{k+\ell}, \, h^\ell(y)=x} B_4^{k+\ell}(y)\,.
\]

By the abovementioned monotonicity each $D^k_\ell$ is an interval.
And as in the paper of David--Mattila these are disjoint sets covering $E$.
The rest of reasoning goes verbatim as in \cite[Theorem 3.2]{MR1768535}.

Therefore we have obtained the following $A_2$-linear estimate for \cz operators with respect to upper doubling measures on $\bR$.
\begin{cor}
\label{A21}
Let $(\bR^1,d,\mu)$ be an upper doubling, geometrically doubling metric measure space, where $d$ is just Euclidean metric,  and let  $T$ be a \cz operator on $\bR^1$, that is, a bounded operator in $L^2(\mu)$ whose kernel satisfies Definition \ref{czONud}.
Then  we have (with $\sigma= w^{-1}$)
\begin{equation}
\label{eq:dep-on-weight}
\norm{T}_{L^{2}(w)\to L^{2}(w)}
\lesssim
\min [\sup_{I\subset \bR^1} \langle w\rangle_{30I}\langle \sigma\rangle_I,\,\sup_{I\subset \bR^1} \langle \sigma\rangle_{30I}\langle w\rangle_I]\,,
\end{equation}
where the supremum is taken over the collection of intervals $I\subset\bR$.
\end{cor}

Moreover, the domination result also holds.

\begin{theorem}
Let $(\bR^1,d,\mu)$ be an upper doubling, geometrically doubling metric measure space, where $d$ is just Euclidean metric, and $\alpha\geq 200$.
Then for every $L^{2}(\mu)$ bounded \cz operator $T$ (see Definition \ref{czONud}) on $\bR^1$, every bounded set $X'\subset \bR^1$, and every integrable function $f$ supported on $X'$ we can find sparse families $\cF_{n}$, $n=0,1,\dots$, of usual intervals $I$ and their subintervals $B(I)$ such that, $I\subset 25 B(I)$, and such that the estimate
\[
T^{\sharp}f
\lesssim_{T,\alpha}
\sum_{n=0}^{\infty} 100^{-n} \sum_{I\in \cF_{n}} \frac{\int_{30B(I)} \abs{f} \dif\mu}{\mu(\alpha B(I))} \cdot \1_I
\]
holds pointwise $\mu$-almost everywhere on $X'$.
\end{theorem}

\printbibliography
\end{document}
